\documentclass[11pt]{article}
\usepackage[utf8]{inputenc}
\usepackage[T1]{fontenc}
\usepackage{graphicx}
\usepackage{longtable}
\usepackage{wrapfig}
\usepackage{rotating}
\usepackage[normalem]{ulem}
\usepackage{amsmath}
\usepackage{amssymb}
\usepackage{capt-of}
\usepackage{hyperref}
\usepackage{amsthm}
\theoremstyle{plain}
\newtheorem{theorem}{Theorem}[section]
\newtheorem{lemma}[theorem]{Lemma}
\newtheorem{corollary}[theorem]{Corollary}
\newtheorem*{MT}{Main Theorem}
\usepackage{tikz}
\author{Rafael Villarroel-Flores\\ Universidad Autónoma del Estado de Hidalgo\\ Carretera Pachuca-Tulancingo km. 4.5\\ Pachuca 42184 Hgo. México\\ email: \texttt{rafaelv@uaeh.edu.mx}keywords: iterated clique graphs, convergent graphs\\ Partially supported by CONACYT, grant A1-S-45528.}
\date{\today}
\title{On the clique behavior of graphs of low degree}
\hypersetup{
 pdfauthor={Rafael},
 pdftitle={On the clique behavior of graphs of low degree},
 pdfkeywords={},
 pdfsubject={},
 pdfcreator={Emacs 28.1 (Org mode 9.5.2)}, 
 pdflang={English}}
\begin{document}

\maketitle
\begin{abstract}
To any simple graph \(G\), the clique graph operator \(K\) associates the graph \(K(G)\) which is the intersection graph of the maximal complete subgraphs of \(G\). The iterated clique graphs are defined by \(K^{0}(G)=G\) and \(K^{n}(G)=K(K^{n-1}(G))\) for \(n\geq 1\). If there are \(m<n\) such that \(K^{m}(G)\) is isomorphic to \(K^{n}(G)\) we say that \(G\) is convergent, otherwise, \(G\) is divergent. The first example of a divergent graph was shown by Neumann-Lara in the 1970s, and is the graph of the octahedron. In this paper, we prove that among the connected graphs with maximum degree 4, the octahedron is the only one that is divergent.
\end{abstract}

\section{Introduction}
\label{sec:org77fc68a}

All graphs in this paper are finite and simple. Following \cite{MR0256911-djvu}, a \emph{clique} in a graph \(G\) is a maximal and complete subgraph of \(G\). The \emph{clique graph} of \(G\), denoted by \(K(G)\), is the intersection graph of its cliques. We define the sequence of \emph{iterated clique graphs} by \(K^{0}(G)=G\) and \(K^{n}(G)=K(K^{n-1}(G))\) for \(n\geq 1\). We say that the graph \(G\) is \emph{convergent} if the sequence of iterated clique graphs has, up to isomorphism, a finite number of graphs. This is equivalent to the existence of \(m<n\) such that \(K^{m}(G)\) is isomorphic to \(K^{n}(G)\) and to the fact that the set of orders of the iterated clique graphs is bounded. If \(G\) is not convergent we say that \(G\) is \emph{divergent}. The graph of the octahedron was the first known example of a divergent graph, given by Neumann-Lara in \cite{zbMATH03641500}.

The problem of determining the \emph{behavior} of \(G\) under iterated applications of the clique operator is one of the main topics in this theory, as the clique graph operator is considered one of the most complex graph operators (\cite{MR1379114-djvu}). There are many families of graphs for which criterions have been proved, and in some cases the behavior can even be determined in polynomial time (see for example \cite{1999-larrion-neumann-lara-clique-divergent-graphs-with-unbounded-sequence-of-diameters}, \cite{MR1746456}, \cite{MR2193079}, \cite{MR2489268}, \cite{MR2002076}, \cite{chessboard-graphs}, \cite{MR2652002}, \cite{2004-larrion-mello-morgana-neumann-lara-pizana-the-clique-operator-on-cographs-and-serial-graphs}, \cite{2013-frias-armenta-larrion-neumann-lara-pizana-edge-contraction-and-edge-removal-on-iterated-clique-graphs}, \cite{2021-baumeister-limbach-clique-dynamics-of-locally-cyclic-graphs-with-elta-geq-6}), however the problem has been found to be undecidable for automatic graphs (\cite{2021-cedillo-pizana-clique-convergence-is-undecidable-for-automatic-graphs}).

We denote the maximum degree in \(G\) as \(\Delta(G)\). The purpose of this paper is to show the following:

\begin{MT}
If \(G\) is a connected graph with \(\Delta(G)\leq 4\) and \(G\) is not the graph of the octahedron, then \(G\) is convergent.
\end{MT}

We use standard notation for graph theory, such as \(x\sim y\) whenever the vertices \(x,y\) are adjacent, \(N(x)\) for the set of neighbors of the vertex \(x\), \(N[x]=N(x)\cup\{x\}\), and \(|A|\) for the cardinality of the set \(A\). Throughout the paper, we usually identify a subset of the vertex set of a graph with the subgraph that it induces. When we say that two sets \emph{meet}, we mean that they have non-empty intersection. A complete subgraph will be called just a \emph{complete}. The difference of the sets \(A,B\) is denoted as \(A-B\). 

\section{Preliminaries}
\label{sec:org5b29f54}

We say that a collection of sets is \emph{intersecting} if any two members of the collection has nonempty intersection. The collection has the \emph{Helly property} if any intersecting subcollection has nonempty intersection. A graph \(G\) is \emph{Helly} if the collection of the cliques of \(G\) has the Helly property. 

The following theorem gives a criterion for convergence.

\begin{theorem}
\label{helly-convergent}
(\cite{MR0329947}) If a graph \(G\) is Helly, then \(G\) is convergent.
\end{theorem}

Not every connected graph with \(\Delta(G)\leq 4\) different from the octahedron is Helly. However, we will prove that for any such graph, one has that \(K^{2}(G)\) is Helly. In order to do that, we will have to study the vertices of \(K^{2}(G)\), that is, the cliques of cliques of \(G\). We will pick some terminology from \cite{MR2002076}, but for the convenience of the reader, we include all the relevant definitions here.

Given \(x\in G\), we define the \emph{star} of \(x\), denoted by \(x^{*}\), as the set of cliques of \(G\) that contain \(x\), that is, \(x^{ * }=\{q\in K(G)\mid x\in q\}\). This is a complete subgraph of \(K(G)\). If \(x^{ *} \in K^{2}(G)\), we say that \(x\) is a \emph{normal vertex}. If \(Q\in K^{2}(G)\) is such that \(\cap Q=\emptyset\), we say that \(Q\) is a \emph{necktie} of \(G\). Thus, vertices of \(K^{2}(G)\) are partitioned in stars and neckties of \(G\), and \(G\) is Helly precisely when \(G\) has no neckties.

A \emph{triangle} in a graph \(G\) is a complete with three vertices. A triangle \(T\) in \(G\) will be called here an \emph{inner triangle} if \(T\) is a clique of \(G\) and for any of the edges of \(T\) there is a triangle \(T'\ne T\) such that \(T\cap T'\) is precisely that edge. Given an inner triangle \(T\) in \(G\), we define \(Q_{T}=\{q\in K(G)\mid |q\cap T|\geq 2\}\). This is always a complete in \(K(G)\). We will prove that if \(G\) is a graph with \(\Delta(G)\leq 4\), and \(G\) is not the octahedron, then \(Q_{T}\) is always a necktie, and also that for any necktie \(Q\) there is a triangle \(T\) of \(G\), such that \(Q=Q_{T}\). In this case, we will say that \(T\) is the \emph{center} of the necktie \(Q_{T}\), and any other element of the necktie that is not the center, will be called an \emph{ear}.

A graph \(G\) is \emph{hereditary Helly} if all its induced subgraphs are Helly. Hence any hereditary Helly graph is Helly. The strategy of the proof of the Main Theorem will be to prove that if \(G\) satisfies the hypothesis, then \(K^{2}(G)\) is hereditary Helly. We will use the characterization proved in \cite{1993-prisner-hereditary-clique-helly-graphs} as reformulated in \cite{LP08}.

\begin{theorem}
\label{theorem-hch}
(Theorem 2.1 in \cite{1993-prisner-hereditary-clique-helly-graphs}) A graph \(G\) is hereditary Helly if and only if it is compatible with the diagram in Figure \ref{diagram-hch}, in the sense that whenever the graph with the solid edges is a subgraph of \(G\), then one of the dashed edges is induced.
\end{theorem}

\begin{figure}
\centering
\begin{tikzpicture}
    [vertex/.style={circle, fill, draw, inner sep=0pt, minimum size=4pt},
     edge/.style={thin},
     edashed/.style={dashed, thin}]
    \node at (-1,0) [vertex] (a) {};
    \node at (0,0) [vertex] (c') {};
    \node at (1,0) [vertex] (b) {};
    \node at (-1/2,{sqrt(3)/2}) [vertex] (b') {};
    \node at (1/2,{sqrt(3)/2}) [vertex] (a') {};
    \node at (0,{sqrt(3)}) [vertex] (c) {};
    \draw[edge] (a) -- (c') -- (b) -- (a') -- (c) -- (b') -- (a);
    \draw[edge] (a') -- (b') -- (c') -- (a');
    \draw[edashed] (a) -- (a');
    \draw[edashed] (b) -- (b');
    \draw[edashed] (c) -- (c');
  \end{tikzpicture}
\caption{Hajós diagram \label{diagram-hch}}
\end{figure}
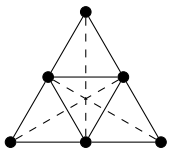

\section{Neckties in low-degree graphs}
\label{sec:org9de2f08}

In what follows, we will call a connected graph \(G\) with \(\Delta(G)\leq 4\) and different from the octahedron, a \emph{low degree graph}.

Note that because of the degree condition on \(G\), the cliques of \(G\) can have at most five vertices. However, the only connected low degree graph that has a clique of five vertices is precisely \(K_{5}\), which is convergent. Hence, for our purposes we may assume that all cliques of \(G\) have at most four vertices. Since the graphs that have no neckties are Helly and hence also convergent, we might as well suppose that \(G\) has neckties.

\begin{lemma}
\label{complete-in-union}
Let \(G\) be any graph, and \(q_{1},q_{2}\) be cliques in \(G\) such that if \(\{x,y\}\) is any edge between vertices in \(q_{1}\cup q_{2}\), with \(x\in q_{1}\) and \(y\in q_{2}\), then one of the vertices \(x,y\) is in \(q_{1}\cap q_{2}\). Then if \(c\) is a complete subgraph such that \(c\subseteq q_{1}\cup q_{2}\), then \(c\subseteq q_{1}\) or \(c\subseteq q_{2}\).
\end{lemma}

\begin{proof}
With the given hypothesis, suppose that there is \(x\in c\) such that \(x\not \in q_{1}\) and  that there is \(y\in c\) such that \(y\not\in q_{2}\). Then \(\{x,y\}\) is an edge with \(x\in q_{1}\) and \(y\in q_{2}\), such that \(\{x,y\}\cap (q_{1}\cap q_{2})=\emptyset\). 
\end{proof}

Note that the hypothesis of Lemma \ref{complete-in-union} is satisfied in the case that \(|q_{1}-q_{2}|=1\). Because in that case, if \(\{x\}=q_{1}-q_{2}\), and there was an edge from \(x\) to the vertex \(y\in q_{2}-q_{1}\), then \(q_{1}\cup\{y\}\) would be a complete subgraph properly containing \(q_{1}\).

\begin{lemma}
\label{only-triangles}
If \(G\) is a graph with \(\Delta(G)\leq 4\), then the neckties of \(G\) can contain only triangles.
\end{lemma}

\begin{proof}
Let \(Q\) be a necktie in \(G\). If \(q_{1}\in Q\) was an edge \(\{x,y\}\), choose \(q_{2}\in Q\) such that \(y\not\in q_{2}\) and \(q_{3}\in Q\) such that \(x\not\in q_{3}\). Then \(q_{1}\cap q_{2}=\{x\}\) and \(q_{1}\cap q_{3}=\{y\}\). Let \(z\in q_{2}\cap q_{3}\). Then \(\{x,y,z\}\) would be a complete subgraph properly containing \(q_{1}\), which is a contradiction.

Now, suppose \(q_{1}\in Q\) is such that \(|q_{1}|=4\). If there was \(q'\in Q\) such that \(|q_{1}\cap q'|=1\), then, since by the previous paragraph, we have \(|q'|\geq 3\), we would have that if \(x\in q_{1}\cap q'\), then \(N[x]\supseteq q_{1}\cup q'\), and this last set has at least six elements, against our hypothesis. Suppose now that there is \(q_{2}\in Q\) with \(|q_{1}\cap q_{2}|=3\). Then \(|q_{2}|=4\). Suppose that \(q_{1}\cap q_{2}=\{x,y,z\}\). Then there is \(q_{3}\in Q\) such that \(x\not\in q_{3}\). If \(q_{3}\) were disjoint to \(q_{1}\cap q_{2}\), then \(q_{1}\cup q_{2}\) would be a complete of five vertices. Hence we may assume \(y\in q_{3}\). Then \(q_{3}\subseteq N[y]=q_{1}\cup q_{2}\), but by Lemma \ref{complete-in-union} this would imply that either \(q_{3}\subseteq q_{1}\) or that \(q_{3}\subseteq q_{2}\), which is a contradiction.

Hence every clique in \(Q\) other than \(q_{1}\) meets \(q_{1}\) in exactly two elements, and has exactly three elements. Let \(q_{2}\in Q\) such that \(q_{2}\ne q_{1}\). If there was \(q_{3}\in Q\) such that \(q_{3}\cap q_{1}\) is disjoint to \(q_{2}\cap q_{1}\), we would have  \(q_{1}=(q_{1}\cap q_{2})\cup (q_{1}\cap q_{3})\), and if \(z\in q_{2}\cap q_{3}\), then \(z\not\in q_{1}\), but \(q_{1}\cup\{z\}\) is a complete subgraph, contradicting our choice of \(q_{1}\). Hence all cliques in \(Q\) meet \(q_{1}\cap q_{2}\). Suppose \(q_{1}=\{x,y,u,v\}\) and \(q_{2}=\{x,y,w\}\). Let \(q_{3}\in Q\) such that \(y\not\in q_{3}\). Then \(x\in q_{3}\), and if there was \(z'\in q_{3}-(q_{1}\cup q_{2})\), we would have \(\{y,u,v,w,z'\}\subseteq N(x)\), which is a contradiction. Hence \(q_{3}\subseteq q_{1}\cup q_{2}\). By Lemma \ref{complete-in-union}, we would have that \(q_{3}\subseteq q_{1}\) or \(q_{3}\subseteq q_{2}\), which is also a contradiction.
\end{proof}

\begin{lemma}
If \(G\) is a low-degree graph, for any necktie \(Q\) in \(G\) there is an inner triangle \(T\) of \(G\) such that \(Q=Q_{T}\), moreover \(Q_{T}\) has exactly three ears. Conversely, if \(T\) is an inner triangle, then \(Q_{T}\) is indeed a necktie.
\end{lemma}

\begin{proof}
We adapt the proof of Proposition 10 in \cite{MR2002076} to this situation. For \(G\) a low-degree graph, let \(Q\) be a necktie of \(G\). From Lemma \ref{only-triangles}, we know that \(Q\) contains only triangles. Let \(T\in Q\). We claim that there must be a triangle in \(Q\) that meets \(T\) in only one vertex. Suppose every other triangle in \(Q\) meets \(T\) at an edge. Let \(T_{1}\in Q\), with \(T_{1}\ne T\) and suppose \(T=\{a,b,c\}\) and \(T_{1}=\{a,b,c'\}\). Now, there must be \(T_{2}\in Q\) such that \(b\not\in T_{2}\). But then \(T_{2}\cap T=\{a,c\}\) and \(T_{2}\cap T_{1}=\{a,c'\}\), which implies \(T_{2}\subseteq T\cup T_{1}\). By Lemma \ref{complete-in-union}, this is not possible.

Let then \(T=\{a,b,c\}\in Q\), and \(T_{1}=\{a,b',c'\}\) with \(T\cap T_{1}=\{a\}\). Choose \(T_{2}\in Q\) such that \(a\not\in T_{2}\). Then \(T_{2}\cap T\) must consist of exactly one vertex, since if \(T_{2}\cap T=\{b,c\}\), then \(T\) would be properly contained in a complete graph of four vertices. Similarly, \(T_{2}\cap T_{1}\) consists of only one vertex. Without loss, assume that \(T_{2}=\{a',b,c'\}\), with \(a'\not\in T\cup T_{1}\). Then \(T_{3}=\{a,b,c'\}\) is a triangle of \(G\). It must be a clique of \(G\) because if, for example, was contained properly in \(\{a,b,c',z\}\), we would have that \(\{b,b',c,c',z\}\subseteq N(a)\), or if, say, we had \(z=c\), then \(T\) would be contained in a tetrahedron. Hence \(T_{3}\) is an inner triangle.

We claim that \(T_{3}\in Q\). If this were not so, there should be \(T_{4}\in Q\) such that \(T_{4}\cap T_{3}=\emptyset\), but that meets \(T\), \(T_{1}\) and \(T_{2}\). We must have then that \(T_{4}\cap T=\{c\}\), \(T_{4}\cap T_{1}=\{b'\}\) and \(T_{4}\cap T_{2}=\{a'\}\). This implies that \(T_{4}=\{a',b',c\}\), and that the subgraph of \(G\) induced by \(\{a,a',b,b',c,c'\}\) is an octahedron. This contradicts our assumption on \(G\), and so this proves our claim that \(T_{3}\in Q\).

We have now that \(\{T,T_{1},T_{2},T_{3}\}\subseteq Q\). Suppose there was a triangle \(T_{4}\) of \(G\), with \(T_{4}\in Q\), and different from \(T,T_{1},T_{2},T_{3}\). We must have \(T_{4}\cap T_{3}\ne\emptyset\), without loss, assume that \(a\in T_{4}\cap T_{3}\). Suppose that \(T_{4}=\{a,r,s\}\). But then \(T_{4}\subseteq N[a]=\{a,b,c,b',c'\}\). It is immediate to verify that any choice of \(r,s\) among \(b,c,b',c'\) leads to contradict either that \(T_{4}\) is different from all of \(T,T_{1},T_{2},T_{3}\), that each of these are cliques, or that \(T_{4}\cap T_{2}\ne\emptyset\). Hence \(Q=\{T,T_{1},T_{2},T_{3}\}=Q_{T_{3}}\), where \(T,T_{1},T_{2}\) are the ears of \(Q_{T_{3}}\).

Now, for the converse, let \(T=\{a,b,c\}\) be an inner triangle. Each of the cliques in \(Q_{T}\) contains at least two vertices of \(T\), so \(Q_{T}\) is a complete in \(K(G)\). Since \(T\) is an internal triangle, there are cliques \(q_{1}, q_{2}, q_{3}\) in \(G\), each different from \(T\) and such that \(\{a,b\}\subseteq q_{1}\cap T\), \(\{a,c\}\subseteq q_{2}\cap T\) and \(\{b,c\}\in q_{3}\cap T\). Now, since \(|T|=3\), such inclusions must in fact be equalities, which implies \(q_{1}, q_{2}, q_{3}\) are three different cliques. Let \(c'\in q_{1}-T\), \(b'\in q_{2}-T\) and \(a'\in q_{3}-T\). Since \(T\) is a clique it follows that \(c'\) is not adjacent to \(c\), and similarly, \(b'\) is not adjacent to \(b\) and \(a'\) is not adjacent to \(a\). This implies that \(\{a',b',c'\}\) are three different vertices. Now, we have that \(\{b,c,b',c'\}\subseteq N(a)\), and given that \(G\) is a low degree graph, this inclusion is equality. It follows then that \(q_{1}=\{a,b,c'\}\), \(q_{2}=\{a,c,b'\}\) and similarly, \(q_{3}=\{b,c,a'\}\). We have now that \(\{T,q_{1},q_{2},q_{3}\}\subseteq Q_{T}\). Suppose \(q\in Q_{T}\) with \(q\ne T\). Without loss, assume that \(q\cap T=\{a,b\}\). Then \(q\subseteq N[a]\cap N[b]=\{a,b,c,c'\}\). Since \(q\ne T\) we have that \(c\not\in q\) and so \(q=\{a,b,c'\}=q_{1}\). It follows then that \(Q_{T}=\{T, q_{1}, q_{2}, q_{3}\}\). Suppose there was \(q\in K(G)\) such that \(Q_{T}\cup\{q\}\) is a complete subgraph of \(K(G)\) with \(q\not\in Q_{T}\). Then \(|q\cap T|=1\), and without loss, assume \(q\cap T=\{a\}\). Then \(q\subseteq N[a]=\{a,b,c,b',c'\}\). On the other hand, must have \(q\cap q_{3}\ne\emptyset\). Since \(|q\cap T|=1\), we have that neither of \(b,c\) could be an element of \(q\). It would follow that \(a'\in q\cap q_{3}\). But this contradicts that \(a\) is not adjacent to \(a'\). It follows then that \(Q_{T}\) is a clique of \(K(G)\), in fact, it is a necktie.
\end{proof}

\section{Forbidden subgraphs}
\label{sec:org21a9da1}

We will apply Theorem \ref{theorem-hch} to \(K^{2}(G)\), but we will first state a series of lemmas to help with the several cases that will arise. Some facts that were proven in \cite{MR2002076} in the context of graphs of local girth at least 7 are true also in the case of low degree graphs.

\begin{lemma}
\label{conditions-adjacent-neckties}
For any low degree graph \(G\):
\begin{enumerate}
\item Two stars \(x ^ {*}\) and \(y *\) are adjacent in \(K^{2}(G)\) if and only if \(x\sim y\) in \(G\).
\item A star \(x^{*}\) and a necktie \(Q_{T}\) centered at the triangle \(T\) are adjacent in \(K^{2}(G)\) if and only if \(x\) is adjacent to two vertices of \(T\).
\item Two neckties \(Q_{T}\), \(Q_{T'}\) are adjacent in \(K^{2}(G)\) if and only if:
\begin{itemize}
\item \(|T\cap T'|=2\), or
\item \(|T\cap T'|=1\) and if \(v\in T\cap T'\), there is an edge (called a \emph{crossbar}) joining a vertex in \(T-v\) to a vertex in \(T'-v\).
\end{itemize}
\end{enumerate}
\end{lemma}

\begin{proof}
In fact, it is immediate to prove that the first two statements are true in any simple graph. Now, for the third, if \(|T\cap T'|=2\), then \(T\in Q_{T}\cap Q_{T'}\), and so \(Q_{T}\) is adjacent to \(Q_{T'}\) in \(K^{2}(G)\) in this case. In the case that \(T\cap T'=\{v\}\) and there is an edge \(\{u,u'\}\) with \(u\in T-v\) and \(u'\in T'-v\), then the triangle \(\{v,u,u'\}\) can be extended to a clique \(q\) such that \(q\in Q_{T}\cap Q_{T'}\).

Now, suppose \(T,T'\) are internal triangles in a low degree graph such that \(Q_{T}\) is adjacent to \(Q_{T'}\) in \(K^{2}(G)\), and that \(|T\cap T'|\leq 1\). Let \(q\in Q_{T}\cap Q_{T'}\). By Lemma \ref{only-triangles}, we have that \(|q|=3\). Since \(|q\cap T|\geq 2\) and \(|q\cap T'|\geq 2\), we have that \(T\cap T'\) cannot be empty, so that \(|T\cap T'|=1\). Let \(T\cap T'=\{v\}\). If \(x\in (q\cap T)-v\) and \(x'\in (q\cap T')-v\), then \(\{x,x'\}\) is a crossbar.
\end{proof}

\begin{lemma}
\label{intersecting-inner-triangles}
Let \(G\) be a low-degree graph. If \(T\) and \(T'\) are triangles in \(G\), where \(T\cap T'=\{x\}\) and \(T\) is an inner triangle, then \(N(x)\) induces a 4-cycle. Moreover, if \(T'\) is also inner, then none of the ears of \(Q_{T}\) is an inner triangle, and none of the vertices of \(T\) besides \(x\) is a normal vertex.
\end{lemma}

\begin{proof}
Suppose that \(T = \{x,e,f\}\) and \(T'  =  \{x,g,h\}\). Since the triangle \(T\) is inner, considering the edge \(\{x,e\}\) there must be \(y\in G\) and a triangle \(T_{1}     =\{x,e,y \}\ne T\). We have that \(y\in N(x)    =\{e,f,g,h\}\). Since \(y\ne e\) because \(|T_{1}| = 3\) and \(y\ne f\) because \(T_{1}\ne  T\), we must have that \(y = g\) or \(y = h\). Without loss, assume \(y = g\), hence \(e\sim g\). Now, considering analogously the edge \(\{x,f\}\), and given that \(T\) is a clique of \(G\), we must have \(f\sim h\). Since \(T\) is a clique, we have that \(e\) is not adjacent to \(h\), and also that \(f\) is not adjacent to \(g\). Hence \(N(x)\) induces a \(4\)-cycle. Given that the degree of \(x\) is \(4\), we must have that \(T'\) is a clique. Since \(T\) is inner, then there is a vertex \(u\) such that the triangle \(\{u,e,f\}\) is different from \(T\). Besides, \(u\not\in\{g,h\}\) because \(T\) is a clique.

In the case that \(T'\) is inner, then analogously there is a vertex \(u'\) such that \(\{u',g,h\}\) is different from \(T'\). We have that \(u\ne u'\), otherwise, \(G\) would be \(4\)-regular and an octahedron. Let us consider the triangle \(T_{1}=\{e,g,x\}\), which is an ear in \(Q_{T}\). If \(T_{1}\) were inner, there would be a vertex \(y\) such that \(\{e,g,y\}\) is a triangle different from \(T_{1}\). But \(y\in N(e)\cap   N(g)=\{x\}\), which is not possible. Hence \(T_{1}\) is not inner, and by symmetry, \(\{f,h,x\}\) is not inner either. Now, consider \(T_{2}=\{e,f,u\}\), and suppose there is a vertex \(z\) such that \(\{e,u,z\}\) is a triangle different from \(T_{2}\). Then \(z\in N(e)=\{u,f,x,g\}\). None of these is possible by our assumptions. Hence \(T_{2}\) is not inner. Finally, neither \(e^{*}\) nor \(f^{ * }\) are cliques, since they are properly contained in \(Q_{T}\).
\end{proof}

As a corollary to the two previous lemmas, we have that if \(T_{1}\) and \(T_{2}\) are internal triangles where \(T_{1}\cap T_{2}\ne\emptyset\), then \(Q_{T_{1}}\sim Q_{T_{2}}\) in \(K^{2}(G)\). Because if \(|T_{1}\cap T_{2}|=1\), then there are (two) crossbars from \(T_{1}\) to \(T_{2}\). The situation in this case is pictured in Figure \ref{intersection-in-one}.
\begin{figure}
\centering
\begin{tikzpicture}
    [vertex/.style={circle, fill, draw, inner sep=0pt, minimum size=4pt},
    edge/.style={thin},
    edashed/.style={dashed, thin}]
    \newcommand{\nodecab}[3]{\node at (#1,#2) [vertex] (#3) {};\node at (#3) [above] {\(#3\)};}
    \newcommand{\nodepab}[3]{\node at (#1:#2) [vertex] (#3) {};\node at (#3) [above] {\(#3\)};}
    \newcommand{\nodepbe}[3]{\node at (#1:#2) [vertex] (#3) {};\node at (#3) [below] {\(#3\)};}
    \nodecab{0}{0}{x}
    \nodecab{{sqrt(3)}}{0}{u'}
    \nodecab{{-sqrt(3)}}{0}{u}
    \nodepab{45}{1}{g}
    \nodepab{135}{1}{e}
    \nodepbe{-135}{1}{f}
    \nodepbe{-45}{1}{h}
    \node at ({sqrt(2)/3},0) {\(T_{2}\)};
    \node at ({-sqrt(2)/3},0) {\(T_{1}\)};
    \draw[edge] (x) -- (e) -- (u) -- (f) -- (x) -- (g) -- (u') -- (h) -- (x);
    \draw[edge] (e) -- (g) -- (h) -- (f) -- (e);
\end{tikzpicture}
\caption{Case \(|T_{1}\cap T_{2}|=1\) \label{intersection-in-one}}
\end{figure}
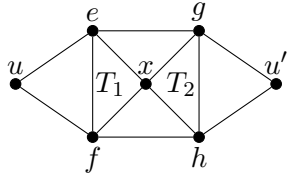

\begin{lemma}
\label{intersecting-inner-in-two}
Let \(G\) be a low-degree graph, and let \(T,T'\) be internal triangles such that \(T\cap T'=\{x,y\}\). Then \(N(x)\) induces a \(4\)-cycle, the only ear of \(Q_{T}\) that is an inner triangle is \(T'\), and the only normal vertices of \(T\) are \(x,y\).
\end{lemma}

\begin{proof}
Suppose that \(T=\{x,y,u\}\), and \(T' = \{x,y,u'\}\). Considering the edge \(\{x,u\}\) of the inner triangle \(T\), let \(x'\in G\) such that \(T_{1} = \{x,u,x'\}\ne T\). Since \(T\) is a clique, we must have \(x'\ne u'\). Applying the Lemma \ref{intersecting-inner-triangles} to the inner triangle \(T'\) and \(T_{1}\), we obtain that \(N(x)\) is a \(4\)-cycle, hence \(x'\sim u'\). Applying the previous steps to the edge \(\{y,u\}\) of the triangle \(T\), we obtain a vertex \(y'\in G\) that is a neighbor of \(u,y,u'\) so that \(N(y)\) is a \(4\)-cycle. Since \(T\) is a clique, we have \(y'\ne x'\).
Now, the triangle \(T_{1}\) is an ear of \(Q_{T}\). If \(T_{1}\) were inner, there would be a vertex \(z\) such that \(\{z,u,x'\}\) is a triangle different from \(T_{1}\). The vertex \(z\) is different from \(y'\) since \(G\) is not an octahedron. Moreover \(z\in N(u)=\{x,x',y,y'\}\), but this contradicts our other assumptions. Hence \(T_{1}\) is not inner. Now, \(u\) is not a normal vertex since \(u^{*}\) is contained properly in \(Q_{T}\). 
\end{proof}

The situation of the previous lemma is shown in Figure \ref{intersection-in-two}.
\begin{figure}
\centering
\begin{tikzpicture}
  [vertex/.style={circle, fill, draw, inner sep=0pt, minimum size=4pt},
  edge/.style={thin},
  edashed/.style={dashed, thin}]
  \newcommand{\nodecab}[3]{\node at (#1,#2) [vertex] (#3) {};\node at (#3) [above] {\(#3\)};}
  \newcommand{\nodepab}[3]{\node at (#1:#2) [vertex] (#3) {};\node at (#3) [above] {\(#3\)};}
  \newcommand{\nodepbe}[3]{\node at (#1:#2) [vertex] (#3) {};\node at (#3) [below] {\(#3\)};}
  \nodecab{0}{0}{x}
  \nodepab{0}{1}{y}
  \nodepab{0}{2}{y'}
  \nodepab{180}{1}{x'}
  \nodepab{60}{1}{u}
  \nodepbe{-60}{1}{u'}
  \node at (30:{sqrt(3)/3}) {\(T_{1}\)};
  \node at (-30:{sqrt(3)/3}) {\(T_{2}\)};
  \draw[edge] (u) -- (x') -- (u') -- (y') -- (u) -- (x) -- (u') -- (y) -- (u);
  \draw[edge] (x') -- (x) -- (y) -- (y');
\end{tikzpicture}
\caption{Case \(|T_{1}\cap T_{2}|=2\) \label{intersection-in-two}}
\end{figure}
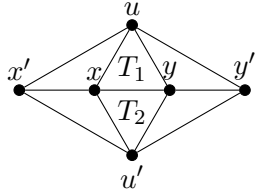

\begin{lemma}
\label{no-path}
If \(G\) is a low-degree graph, there is no path (induced or not) in \(K^{2}(G)\), that consists of three neckties.
\end{lemma}

\begin{proof}
Let \(T_{1}\), \(T_{2}\) and \(T_{3}\) be three inner triangles in \(G\) such that \(Q_{T_{1}}\sim Q_{T_{2}}\sim Q_{T_{3}}\) in \(K^{2}(G)\). From \(Q_{T_{1}}\sim Q_{T_{2}}\), we have one of the two situations shown in either Figure \ref{intersection-in-one} or Figure \ref{intersection-in-two}. But in none of those cases it is possible to have an inner triangle \(T_{3}\), different from \(T_{1}\), and satisfying  \(T_{3}\cap T_{2}\ne\emptyset\).
\end{proof}

\begin{lemma}
Let \(G\) be a low-degree graph, where \(a,b\in G\) are normal vertices and \(T_{1},T_{2}\) are inner triangles. Then \(a^{*}, b^{ * }, Q_{T_{1}}, Q_{T_{2}}\) do not induce a diamond in \(K^{2}(G)\) of one of the forms shown in Figure \ref{forbidden-diamond}.
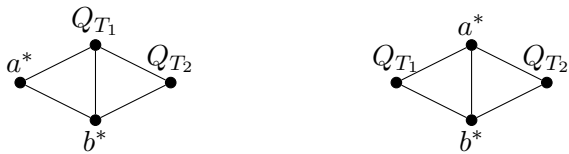
\begin{figure}
\centering
\begin{tikzpicture}
    [vertex/.style={circle, fill, draw, inner sep=0pt, minimum size=4pt},
    edge/.style={thin},
    edashed/.style={dashed, thin}]
    \newcommand{\nodecab}[3]{\node at (#1,#2) [vertex] (#3) {};\node at (#3) [above] {\(#3\)};}
    \newcommand{\nodecbe}[3]{\node at (#1,#2) [vertex] (#3) {};\node at (#3) [below] {\(#3\)};}
    \newcommand{\nodepab}[3]{\node at (#1:#2) [vertex] (#3) {};\node at (#3) [above] {\(#3\)};}
    \newcommand{\nodepbe}[3]{\node at (#1:#2) [vertex] (#3) {};\node at (#3) [below] {\(#3\)};}
    \nodecab{1}{0}{Q_{T_2}}
    \nodecab{0}{1/2}{Q_{T_1}}
    \nodecab{-1}{0}{a^*}
    \nodecbe{0}{-1/2}{b^*}
    \draw[edge] (a^*) -- (Q_{T_1}) -- (Q_{T_2}) -- (b^*) -- (a^*);
    \draw[edge] (Q_{T_1}) -- (b^*);
    \begin{scope}[xshift=5cm]
      \nodecab{1}{0}{Q_{T_2}}
      \nodecab{0}{1/2}{a^*}
      \nodecab{-1}{0}{Q_{T_1}}
      \nodecbe{0}{-1/2}{b^*}
      \draw[edge] (a^*) -- (Q_{T_1}) -- (b^*) -- (Q_{T_2}) -- (a^*);
      \draw[edge] (a^*) -- (b^*);      
    \end{scope}
\end{tikzpicture}
\caption{Forbidden diamonds \label{forbidden-diamond}}
\end{figure}
\end{lemma}

\begin{proof}
Consider the case to the left in Figure \ref{forbidden-diamond}. Since \(Q_{T_{1}}\sim Q_{T_{2}}\), we have \(|T_{1}\cap T_{2}|\ne \emptyset\). If \(|T_{1}\cap T_{2}|=1\), we have the situation of Figure \ref{intersection-in-one}. Since \(b^{*}\) is a neighbor of both \(Q_{T_{1}}\) and \(Q_{T_{2}}\), and \(b\) is a normal vertex, we have that \(b\in T_{1}\cap T_{2}\). It is immediate to verify that no neighbor of \(b\) is a normal vertex. Now, if \(|T_{1}\cap T_{2}|=2\), then again our assumptions on \(b\) imply that \(b\in T_{1}\cap T_{2}\), and all neighbors \(x\) of \(b\) that are normal vertices are such that \(x^{ * }\) is a neighbor of both \(Q_{T_{1}}\) and \(Q_{T_{2}}\).

Now, with respect to the case on the right, since \(Q_{T_{1}}\) is not adjacent to \(Q_{T_{2}}\) in \(K^{2}(G)\), then \(T_{1}\cap T_{2}=\emptyset\). Since \(b\) must be a neighbor to two vertices in \(T_{1}\), to two vertices in \(T_{2}\) and to the vertex \(a\), and the degree of \(b\) is at most four, we have that \(a\in T_{1}\cup T_{2}\). Without loss, assume \(a\in T_{1}\). By a similar argument, we must have \(b\in T_{2}\). Let \(T_{1}=\{a,u,v\}\) and \(T_{2}=\{b,w,r\}\). Since \(a\) is a neighbor to two vertices in \(T_{2}\), without loss we may assume that \(a\sim w\). Applying Lemma \ref{intersecting-inner-triangles} to the inner triangle \(T_{1}\) and the triangle \(T=\{a,b,w\}\), we have that \(N(a)\) is a \(4\)-cycle, and without loss, we may assume \(u\sim b\) and \(v\sim w\). Similarly, applying the same lemma to \(T_{2}\) and the triangle \(T' = \{a,b,u\}\) we have that \(u\sim r\). But then \(T\)  and \(T'\) are inner triangles, and Lemma \ref{intersecting-inner-in-two} applies. We have that \(T\) is the only ear in \(Q_{T'}\) that is an inner triangle, which contradicts that \(T_{1}\) is inner.
\end{proof}

\begin{lemma}
\label{lemma-III}
Let \(G\) be a low-degree graph, and let \(a,b,x,z\in G\) be normal vertices and \(T\) an inner triangle in \(G\) forming the subgraph of \(K^{2}(G)\) with the solid edges in Figure \ref{case-III}. Then one of the dashed edges has to be present.

\begin{figure}
\centering
\begin{tikzpicture}
  [vertex/.style={circle, fill, draw, inner sep=0pt, minimum size=4pt},
  edge/.style={thin},
  edashed/.style={dashed, thin}]
  \newcommand{\nodecab}[3]{\node at (#1,#2) [vertex] (#3) {};\node at (#3) [above] {\(#3\)};}
  \newcommand{\nodecbe}[3]{\node at (#1,#2) [vertex] (#3) {};\node at (#3) [below] {\(#3\)};}
  \newcommand{\nodepab}[3]{\node at (#1:#2) [vertex] (#3) {};\node at (#3) [above] {\(#3\)};}
  \newcommand{\nodepbe}[3]{\node at (#1:#2) [vertex] (#3) {};\node at (#3) [below] {\(#3\)};}
  \nodecbe{0}{0}{b^*}
  \nodepbe{0}{1}{x^*}
  \nodepab{60}{1}{Q_T}
  \nodepab{120}{1}{a^*}
  \nodepbe{180}{1}{z^*}
  \draw[edge] (b^*) -- (z^*) -- (a^*) -- (Q_T) -- (x^*) -- (b^*);
  \draw[edge] (a^*) -- (b^*) -- (Q_T);
  \draw[dashed] (z^*) -- (Q_T);
  \draw[dashed] (a^*) -- (x^*);
\end{tikzpicture}
\caption{Forbidden subgraph \label{case-III}}
\end{figure}
\end{lemma}

\begin{proof}
Suppose there is a subgraph with the solid edges of the form in Figure \ref{case-III}, but without the dashed edges. Then in \(G\) we have that \(\{a,x.z\}\subseteq N(b)\). Since \(b^{ * }\sim Q_{T}\) in \(K^{2}(G)\), we have that two neighbors of \(b\) are in \(T\), and so, at least one of \(\{a,x,z\}\) is an element of \(T\). But since \(z\not\sim Q_{T}\), we have that \(z\not\in T\), and it is not possible for \(a,b\) to be both elements of \(T\).

Suppose first that \(a\in T\) (and so, \(b\not\in T\)). Let us write \(T=\{a,u,v\}\).  Now, \(b\) is a neighbor of two elements of \(T\), let us say without loss that \(b\sim v\). Applying Lemma \ref{intersecting-inner-triangles}, we must then have that \(z\sim u\). But this contradicts that \(z\not\sim Q_{T}\). Consider then the case that \(a\not\in T\) and \(x\in T\). Suppose \(T=\{x,u,v\}\). If we had that \(b\in\{u,v\}\), we could assume without loss that \(b=v\). Since \(a\) is a neighbor of two elements of \(T\), and \(a\not\sim x\) by our hypothesis, then \(a\sim u\). Applying Lemma \ref{intersecting-inner-triangles} to the inner triangle \(T\) and the triangle \(\{b,a,z\}\), we obtain that \(z\sim x\), which again would mean that \(z\) is a neighbor of two elements of \(T\). This shows that \(b\not\in\{u,v\}\). Then \(a^{ * }\sim Q_{T}\) in \(K^{2}(G)\) implies that \(a\sim u\) and \(a\sim v\) in \(G\). We also have \(b^{ * }\sim Q_{T}\) in \(K^{2}(G)\), without loss, assume \(b\sim v\). Then the triangle \(\{a,b,v\}\) is inner, and we are in the situation of Figure \ref{intersecting-inner-triangles}, which shows that \(x\) is not a normal vertex, contradicting our hypothesis. 
\end{proof}

\begin{lemma}
\label{lemma-IV}
Let \(G\) be a low-degree graph, and let \(a,b,x,z\in G\) be normal vertices and \(T\) an inner triangle in \(G\) forming the subgraph of \(K^{2}(G)\) with the solid lines in Figure \ref{case-IV}. Then one of the dashed lines has to be present.

\begin{figure}
\centering
\begin{tikzpicture}
  [vertex/.style={circle, fill, draw, inner sep=0pt, minimum size=4pt},
  edge/.style={thin},
  edashed/.style={dashed, thin}]
  \newcommand{\nodecab}[3]{\node at (#1,#2) [vertex] (#3) {};\node at (#3) [above] {\(#3\)};}
  \newcommand{\nodecbe}[3]{\node at (#1,#2) [vertex] (#3) {};\node at (#3) [below] {\(#3\)};}
  \newcommand{\nodepab}[3]{\node at (#1:#2) [vertex] (#3) {};\node at (#3) [above] {\(#3\)};}
  \newcommand{\nodepbe}[3]{\node at (#1:#2) [vertex] (#3) {};\node at (#3) [below] {\(#3\)};}
  \nodecbe{0}{0}{b^*}
  \nodepbe{0}{1}{x^*}
  \nodepab{60}{1}{c^*}
  \nodepab{120}{1}{a^*}
  \nodepbe{180}{1}{Q_T}
  \draw[edge] (b^*) -- (Q_T) -- (a^*) -- (c^*) -- (x^*) -- (b^*);
  \draw[edge] (a^*) -- (b^*) -- (c^*);
  \draw[dashed] (c^*) -- (Q_T);
  \draw[dashed] (a^*) -- (x^*);
\end{tikzpicture}
\caption{Forbidden subgraph \label{case-IV}}
\end{figure}
\end{lemma}

\begin{proof}
Suppose there is a subgraph with the solid edges of the form in Figure \ref{case-IV}, but without the dashed edges. Then in \(G\) we have that the normal vertices \(\{a,c,b,x\}\) induce a diamond. Since \(b^{ * }\sim Q_{T}\) and \(\{a,c,x\}\subseteq N(b)\), at least one of \(\{a,c,x\}\) is in the inner triangle \(T\). But since \(c^{ * }\not\sim Q_{T}\), we have that \(c\not\in T\), and at most one of \(a,b\) is in \(T\). Consider first the case where \(a\in T\), and let \(T=\{a,u,v\}\). Since \(c^{ * }\not\sim Q_{T}\) we must have \(\{u,v\}\cap \{b,c,x\}=\emptyset\). Applying Lemma \ref{intersecting-inner-triangles}, to the inner triangle \(T\) and the triangle \(\{a,b,c\}\), we have that \(c\) would have to be adjacent to either of \(u\) or \(v\), but this contradicts \(c^{ * }\not\sim Q_{T}\). So, let us consider the case \(a\not \in T\) and \(x\in T\). Then \(T=\{x,u,v\}\) with \(\{u,v\}\cap \{a,b,c\}=\emptyset\). As before, an application of Lemma \ref{intersecting-inner-triangles} leads to \(c\sim u\), or to \(c\sim v\) but this is a contradiction.
\end{proof}

\section{Proof of the Main Theorem}
\label{sec:orgcfbcd0b}

\begin{theorem}
Let \(G\) be a low-degree graph. If the graph with the solid edges in Figure \ref{all-complete} appears as a subgraph of \(K^{2}(G)\), then there is necessarily at least one of the dotted edges.

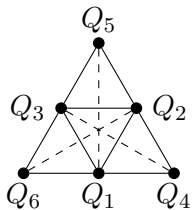
\begin{figure}
\centering
\begin{tikzpicture}
  [vertex/.style={circle, fill, draw, inner sep=0pt, minimum size=4pt},
  edge/.style={thin},
  edashed/.style={dashed, thin}]
  \newcommand{\nodecab}[3]{\node at (#1,#2) [vertex] (#3) {};\node at (#3) [above] {\(#3\)};}
  \newcommand{\nodecbe}[3]{\node at (#1,#2) [vertex] (#3) {};\node at (#3) [below] {\(#3\)};}
  \newcommand{\nodepab}[3]{\node at (#1:#2) [vertex] (#3) {};\node at (#3) [above] {\(#3\)};}
  \newcommand{\nodepbe}[3]{\node at (#1:#2) [vertex] (#3) {};\node at (#3) [below] {\(#3\)};}
  \newcommand{\nodepang}[4]{\node[label=#4:\(#3\)] at (#1:#2) [vertex] (#3) {};}
  \nodecbe{0}{0}{Q_1}
  \nodepbe{0}{1}{Q_4}
  \nodepang{60}{1}{Q_2}{0}
  \nodepang{120}{1}{Q_3}{180}
  \nodepbe{180}{1}{Q_6}
  \nodepab{90}{{sqrt(3)}}{Q_5}
  \draw[edge] (Q_1) -- (Q_4) -- (Q_2) -- (Q_5) -- (Q_3) -- (Q_6) -- (Q_1);
  \draw[edge] (Q_1) -- (Q_2) -- (Q_3) -- (Q_1);
  \draw[dashed] (Q_1) -- (Q_5);
  \draw[dashed] (Q_2) -- (Q_6);
  \draw[dashed] (Q_3) -- (Q_4);
\end{tikzpicture}
\caption{Forbidden subgraph \label{all-complete}}
\end{figure}
\end{theorem}

\begin{proof}
Suppose there is a subgraph with the solid edges of the form in Figure \ref{all-complete}, but without the dashed edges.  We consider several cases:
\begin{enumerate}
\item Suppose first that \(Q_{1},Q_{2},Q_{3}\) are stars, and set  \(Q_{i}=a_{i}^{ * }\) for \(i=1,2,3\). As a subcase, consider when \(Q_{4},Q_{5},Q_{6}\) are also stars. Then there is a subgraph of \(G\) isomorphic to the one in Figure \ref{all-complete}, whose vertices are normal vertices. However, then the triangle \(T=\{a_{1},a_{2},a_{3}\}\) is inner. Since \(a_{1}\) is a normal vertex, there must be \(q\in K(G)\) such that \(q\in a_{1}^{ * }\) but \(q\not\in Q_{T}\). Since \(q\subseteq N[a_{1}]=\{a_{2},a_{3},a_{4},a_{6}\}\), we must have \(q=\{a_{1},a_{4},a_{6}\}\), and so \(a_{4}\) is adjacent to \(a_{6}\). Similarly, considering that \(a_{2}\) is a normal vertex we obtain that \(a_{4}\sim a_{5}\), and considering \(a_{3}\) we obtain that \(a_{5}\sim a_{6}\). But then this would imply that \(G\) is an octahedron. Hence, this subcase is not possible. Now, consider the subcase when one of \(Q_{4},Q_{5},Q_{6}\) is a necktie. Without loss, suppose that \(Q_{4}\) is a necktie. By Lemma \ref{lemma-IV}, neither \(Q_{5}\) nor \(Q_{6}\) are stars. Pick then inner triangles \(T_{1},T_{2},T_{3}\) such that \(Q_{4}=Q_{T_{3}}\), \(Q_{5}=Q_{T_{1}}\) and \(Q_{6}=Q_{T_{2}}\). We claim that one of \(T_{1}\cap T_{2}\), \(T_{2}\cap T_{3}\), \(T_{1}\cap T_{3}\) must be non-empty. Suppose that \(T_{1}\cap T_{3}=\emptyset\) and \(T_{2}\cap T_{3}=\emptyset\). Since \(a_{1}\) is a neighbor of two vertices of \(T_{2}\) and of two vertices of \(T_{3}\), and of \(a_{2}\) and of \(a_{3}\), we must have that \(a_{2}\in T_{3}\) and \(a_{3}\in T_{2}\). Similarly, considering the neighbors of \(a_{2}\), we get that \(a_{1}\in T_{3}\) and \(a_{3}\in T_{1}\). This means that \(T_{1}\cap T_{2}\ne\emptyset\). This proves our claim, so without loss, assume that \(T_{1}\cap T_{2}\ne\emptyset\). If \(|T_{1}\cap T_{2}|=1\), we are in the situation of Figure \ref{intersection-in-one}. In this case, one can verify that the only normal vertex such that its star is a neighbor in \(K^{2}(G)\) of both \(Q_{T_{1}}\) and \(Q_{T_{2}}\) is the vertex in \(T_{1}\cap T_{2}\), then \(T_{1}\cap T_{2}=\{a_{3}\}\). But then none of the neighbors of \(a_{3}\) is a normal vertex. So we are left with the possibility that \(|T_{1}\cap T_{2}|=2\), that is, the situation of Figure \ref{intersection-in-two}. But we obtain that this subcase is not possible either, since, given that \(a_{3}\) is a normal vertex with \(a_{3}^{*}\sim Q_{T_{1}}\) and \(a_{3}^{ * }\sim Q_{T_{2}}\), it must happen that \(a_{3}\in T_{1}\cap T_{2}\). But then \(a_{3}\) cannot have two different neighbors that are normal vertices.
\item Suppose now that \(Q_{1}\) is a necktie, and \(Q_{2},Q_{3}\) are stars. Considering the graph to the right of Figure \ref{forbidden-diamond}, we obtain that \(Q_{5}\) must be a star. Considering the graph to the left of Figure \ref{forbidden-diamond}, we obtain that \(Q_{6}\) must also be a star. But by Lemma \ref{lemma-III}, this is not possible.
\item Suppose now that \(Q_{1},Q_{2}\) are neckties, and \(Q_{3}\) is a star. By Lemma \ref{no-path}, \(Q_{5}\) cannot be a necktie. But by the graph to the left of Figure \ref{forbidden-diamond}, \(Q_{5}\) cannot be a star either.
\item The case when \(Q_{1},Q_{2},Q_{3}\) are all neckties, is also impossible by Lemma \ref{no-path}.
\end{enumerate}

This finishes the proof of the Main Theorem.
\end{proof}

\begin{corollary}
If \(G\) is a connected, divergent graph, that is different from the octahedron, then \(G\) has a vertex with degree at least five.
\end{corollary}

\bibliographystyle{plain}
\bibliography{lowdegree}
\end{document}